\documentclass{llncs}

\usepackage{llncsdoc}

\usepackage{amssymb}
\usepackage{bm}
\usepackage{amsfonts}
\usepackage{amsmath}
\usepackage{graphicx}
\usepackage{psfrag}
\usepackage{color}
\usepackage{url}
\usepackage{enumerate}

\newtheorem{thm}{Theorem}
\newtheorem{cor}[thm]{Corollary}

\newcommand{\NC}{{Q}}
\newcommand{\Q}{{\mathbb Q}}
\newcommand{\R}{{\mathbb R}}
\newcommand{\RR}{{R}}
\newcommand{\Z}{{\mathbb Z}}

\newcommand{\be}{\begin{eqnarray}}
\newcommand{\bea}{\begin{eqnarray*}}

\newcommand{\ee}{\end{eqnarray}}
\newcommand{\eea}{\end{eqnarray*}}
\newcommand{\gap}{\mathrm{Gap}}

\newcommand{\frob}{g}
\newcommand{\ffrob}{f}
\renewcommand{\mod}{\;\mathrm{mod}\;}

\newcommand{\ve}{\boldsymbol}

\newcommand{\A}{A}
\newcommand{\modulo}{\,\mathrm{mod}\;}
\newcommand{\sign}{\mathrm{sign}}
\newcommand{\idist}{{d}}

\begin{document}

\title{Distances to Lattice Points in Knapsack Polyhedra \thanks{Earlier proceeding version: Iskander Aliev, Martin Henk and Timm Oertel, {\em Integrality Gaps of Integer Knapsack Problems}, Integer Programming and Combinatorial Optimization, Lecture Notes in Computer Science, {\bf 10328} (2017), 25--38.}}

\author{Iskander Aliev\inst{1} \and Martin Henk\inst{2} \and Timm Oertel\inst{3}}
\institute{Mathematics Institute, Cardiff University, UK\\
\email{alievi@cardiff.ac.uk}
\and
Department of Mathematics, TU Berlin, Germany\\
\email{henk@math.tu-berlin.de}\\
\and
Mathematics Institute, Cardiff University, UK\\
\email{oertelt@cardiff.ac.uk}}

\date{\today}

\maketitle

\begin{abstract}
We give an optimal upper bound for the $\ell_{\infty}$-distance from a vertex of a knapsack polyhedron to its nearest feasible lattice point. In a randomised setting, we show that the upper bound can be significantly improved on average. As a corollary, we obtain an optimal upper bound for the additive integrality gap of integer knapsack problems and show that the integrality gap of  a ``typical'' knapsack problem is drastically smaller than the integrality gap that occurs in a worst case scenario. We also prove that, in a generic case, the integer programming gap admits a natural optimal lower bound.

\end{abstract}

\section{Introduction}

%
%

Given  ${\ve a}\in \Z^n$, $b\in \Z$, a {\em knapsack polyhedron} $P({\ve a}, b)$  is defined as
\bea
P({\ve a}, b)=\{{\ve x}\in \R^n_{\ge 0}: {\ve a}^T {\ve x}=b\}\,.
\eea
We will estimate the $\ell_\infty$-distance from a vertex of $P({\ve a}, b)$ to the set of its lattice points. 
For this purpose we define the {\em (maximum) vertex distance}
\bea
\idist({\ve a},b)=
\left\{
\begin{array}{rl}
\max_{{\ve v}}\min_{{\ve z}\in P({\ve a},b)\cap \Z^n} \|{\ve v}-{\ve z}\|_{\infty}\,, & \mbox{ if } P({\ve a},b)\cap \Z^n\neq \emptyset\,,\\
-\infty\,, & \mbox{ otherwise}\,,
\end{array}
\right.
\eea
where $\|{\cdot}\|_{\infty}$ stands for the $\ell_\infty$-norm and the maximum is taken over all vertices ${\ve v}$ of the polyhedron $P({\ve a},b)$.

We will exclude the trivial case $n=1$, where the vertex distance takes the values $0$ and $-\infty$ only. We may also assume without loss of generality
that ${\ve a}$ is a primitive integer vector with nonzero entries.
Thus, we will assume the following conditions:
\be \label{nonzeroA}
\begin{array}{ll}
(i) & {\ve a}=(a_1, \ldots, a_n)^T\in \Z^{ n}\,,n\ge 2\,, a_i\neq 0\,, i=1,\ldots,n\,,\\
(ii) & \gcd({\ve a}):=\gcd(a_1, \ldots, a_n)=1\,.
\end{array}
\ee

The first result of this paper gives an optimal upper bound for the vertex distance that depends only on the $\ell_\infty$-norm of the vector ${\ve a}$ and independent of $n$ and $b$.
\begin{thm}\label{thm_bounds_for_distance}
\begin{itemize}
\item[(i)]
Let ${\ve a}$ satisfy (\ref{nonzeroA}) and $b\in \Z$. Then
\be\label{thm_upper}
\idist({\ve a},b)\le\|{\ve a}\|_\infty-1\,.
\ee
\item[(ii)]
For any positive integer $k$ and any dimension $n$ there exist ${\ve a}$ satisfying (\ref{nonzeroA}) with $\|{\ve a}\|_{\infty}=k$  and $b\in \Z$ such that
\bea
\idist({\ve a},b)=\|{\ve a}\|_\infty-1\,.
\eea
\end{itemize}
\end{thm}

Note that the classical sensitivity theorem of Cook et al. \cite[Theorem 1]{Cook} implies
in the knapsack setting the bound
$\idist({\ve a},b) \le n \|{\ve a}\|_\infty.$
Let $A=(a_{ij})\in\Z^{m \times n}$, ${\ve b}\in\Z^m$ and let $\|\cdot\|_1$ denote the $l_1$-norm. A very recent strong improvement on the results of Cook et al. \cite{Cook} obtained by  Eisenbrand and Weismantel \cite{EW} implies that to every vertex ${\ve v}$ of the polyhedron $P=\{{\ve x}\in \R^n_{\ge 0}: A{\ve x}={\ve b}\}$ there exists an integer point ${\ve z}$ in $P$ (provided it is integer feasible), such that
\be\label{EW_implies}
\|{\ve v}-{\ve z}\|_1\le m(2m\|A\|_\infty+1)^m,
\ee
where $\|A\|_\infty=\max_{i,j} |a_{ij}|$.
It remains an open question how tight this bound is.
For a bounded knapsack polyhedron the bound (\ref{EW_implies}) can be strengthened as follows. In the proof of Theorem \ref{thm_bounds_for_distance} (i) we estimate the vertex distance using a covering argument that guarantees for any vertex ${\ve v}$ of a bounded polyhedron $P({\ve a},b)$ existence of a lattice point ${\ve z}\in P({\ve a},b)$ in an $(n-1)$-dimensional simplex of sufficiently small diameter, translated by ${\ve v}$.
The argument implies the bound
\bea 
\min_{{\ve z}\in P({\ve a},b)\cap \Z^n} \|{\ve v}-{\ve z}\|_1\le  2 (\|{\ve a}\|_{\infty}-1)\,.
\eea



How large is the  vertex distance of a ``typical'' knapsack polyhedron? Specifically, consider for $H\ge 1$
the set ${\NC}(H)$ of ${\ve a}\in \Z^{ n}$ that satisfy (\ref{nonzeroA}) and
\bea \|{\ve a}\|_{\infty}\le H\,.\eea
The next theorem will estimate the proportion of the vectors ${\ve a}$ in ${\NC}(H)$ such that for some $b\in\Z$ the knapsack polyhedron
$P({\ve a}, b)$ has relatively large vertex distance.
Let
$N(H)$ be the cardinality of ${\NC}(H)$. For $\epsilon\in (0,3/4)$ let
\bea
 N_{\epsilon}(t,H)=\#\left\{{\ve a}\in {\NC}(H): \max_{b\in\Z}\frac{\idist({\ve a},b)}{\|{\ve a}\|_{\infty}^{\epsilon}}>t \right\}\,.
\eea
In the rest of the paper, the notation  $f({\ve x})\ll_n g({\ve x})$ for ${\ve x}\in S$, where $S$ is a set, means that $|f({\ve x})|\le c|g({\ve x})|$, for ${\ve x}\in S$ and a positive constant $c=c(n)$ depending on $n$ only.  $f({\ve x})\asymp_n g({\ve x})$ means that both $f({\ve x})\ll_n g({\ve x})$, $g({\ve x})\ll_n f({\ve x})$ hold.

\begin{thm}\label{Ratio} Fix $n\ge 3$. For any $\epsilon\in (0,3/4)$ we have
\be\label{main_bound}
\frac{N_{\epsilon}(t,H)}{N(H)}\ll_n t^{-\alpha(\epsilon,n)}
\ee
uniformly over all $t>0$ and $H\ge 1$. Here
\bea
\alpha(\epsilon,n)=\frac{n-2}{(1-\epsilon)n}\,.
\eea
\end{thm}

To prove Theorem \ref{Ratio}, we will utilize results of Str\"ombergsson \cite{Str} (see also Schmidt \cite{Schmidt} and references therein) on the asymptotic distribution of Frobenius numbers.

Theorems  \ref{thm_bounds_for_distance} and \ref{Ratio} can be applied to estimating the (additive) integrality gaps
 for integer knapsack problems. In the proceedings \cite{AHO} the authors have considered this problem for the case that $\ve a$ is non-negative. In this paper we extend those results to greater generality.
In particular, in Corollary~\ref{coro_upper_bound} and Corollary~\ref{average}, we show that the two main statements of \cite{AHO} hold true for general knapsack polyhedra, i.e., we drop the non-negativity assumption. We remark that extending results of \cite{AHO} to general knapsack polyhedra required using new covering arguments in the proofs of Theorems \ref{thm_bounds_for_distance} and \ref{Ratio}.
Also, we include in this paper proofs that were omitted in \cite{AHO}.

 Given  ${\ve a}\in \Z^n$, $b\in \Z$ and a cost vector ${\ve c}\in \Q^n$, we will consider the integer knapsack problem
 \be
\min\{ {\ve c} ^T{\ve x}: {\ve x}\in P({\ve a},b)\cap \Z^n\}\,.
\label{initial_IP}
\ee
We will assume that (\ref{initial_IP}) is feasible and bounded.

Let $IP({\ve c}, {\ve a}, {b})$ and $LP({\ve c}, {\ve a}, { b})$ denote the optimal values of (\ref{initial_IP}) and its linear programming relaxation
\be\label{initial_LP}
\min\{ {\ve c}^T {\ve x}: {\ve x}\in P({\ve a},b)\}\,,
\ee
 respectively.
The {\em integrality gap } $IG({\ve c}, {\ve a}, {b})$ of  (\ref{initial_IP}) is defined as
\bea IG({\ve c}, {\ve a}, {b})= IP({\ve c}, {\ve a}, { b})-LP({\ve c}, {\ve a}, { b})\,.\eea
Notice that
\be\label{gap-distance}
IG({\ve c}, {\ve a}, {b})\le \idist({\ve a},b)\|{\ve c}\|_1\,.
\ee
%
%
Given a pair $({\ve c}, {\ve a})$, the maximum of $IG({\ve c}, {\ve a}, {b})$ over all suitable ${b}$ is referred to as the {\em integer programming gap} (Ho\c{s}ten and Sturmfels \cite{HS})
\bea
\gap({\ve c}, {\ve a})= \max_{{b}} IG({\ve c}, {\ve a}, {b})\,.
\eea
Here ${ b}$ ranges over all integers such that (\ref{initial_IP}) is feasible and bounded. Notice that computing $\gap({\ve c}, {\ve a})$ when $n$ is a part of input is NP-hard (see Aliev \cite{lpg} and Eisenbrand et al \cite{EHPS}). For any fixed $n$, the integer programming gap  can be computed in polynomial time due to results of Ho\c{s}ten and Sturmfels \cite{HS} (see also Eisenbrand and Shmonin \cite{ES}).

As a corollary of Theorem \ref{thm_bounds_for_distance}, we obtain the following optimal upper bound on the integer programming gap.

\begin{corollary}\label{coro_upper_bound}\hfill \begin{itemize}

\item[(i)] Let ${\ve a}$ satisfy (\ref{nonzeroA}) and let ${\ve c}\in \Q^n$. Then
\be\label{coro_upper}
\gap({\ve c}, {\ve a})\le \left( \|{\ve a}\|_{\infty}-1 \right)\|{\ve c}\|_1\,.
\ee
\item[(ii)] For any positive integer $k$ and any dimension $n$ there exist ${\ve a}$ satisfying (\ref{nonzeroA}) with $\|{\ve a}\|_{\infty}=k$ and ${\ve c}\in \Q^n$ such that
\bea
\gap({\ve c}, {\ve a})= \left( \|{\ve a}\|_{\infty}-1 \right)\|{\ve c}\|_1\,.
\eea

\end{itemize}
\end{corollary}

From (\ref{main_bound}) one can derive an upper  bound on the average
value of the (normalised) integer programming gap. The next corollary will show that for any $\epsilon>2/n$
the ratio
\bea
\frac{\gap({\ve c}, {\ve a})}{\|{\ve a}\|_{\infty}^{\epsilon}\|{\ve c}\|_1}
\eea
is bounded, on average, by a constant that depends only on dimension $n$.
Hence, for fixed $n>2$ and a ``typical'' integer knapsack problem with large $\|{\ve a}\|_{\infty}$,
its linear programming relaxation provides a drastically better approximation (roughly of order $2/n$) to the solution than in the worst case scenario, determined by the optimal upper bound (\ref{coro_upper}).

\begin{corollary} \label{average} Fix $n\geq 3$. For any $\epsilon>2/n$
\be\label{average_ineq}
\frac{1}{N(H)}\sum_{{\ve a}\in \NC(H)}\max_{{\ve c}\in \Q^n}\frac{\gap({\ve c}, {\ve a})}{\|{\ve a}\|_{\infty}^{\epsilon}\|{\ve c}\|_1} \ll_n 1\,.
\ee
\end{corollary}

The last two theorems of this paper give lower bounds for the integer programming gap and its average value. In particular, Theorem \ref{lower_upper} shows that the bound in Corollary \ref{average} is not far from being optimal.

Let ${\ve a}\in \Z^n_{>0}$ satisfy \eqref{nonzeroA} and let ${\ve c}\in \Q^n$. We will say that  $({\ve a}, {\ve c})$ is {\em generic} if for any positive $b \in \Z$
the linear programming relaxation (\ref{initial_LP}) has a unique optimal solution.
In this setting, an optimal lower bound for $\gap({\ve c}, {\ve a})$ 
can be obtained
using recent results \cite{lpg}  on the {\em lattice programming gaps} associated with the group relaxations to (\ref{initial_IP}).
%
%
For a generic $({\ve a},{\ve c})$, let $\tau=\tau({\ve a}, {\ve c})$ be the unique index of the basic variable for the optimal solution to the linear relaxation (\ref{initial_LP}) with a positive $b\in \Z$. Let $\pi_i(\cdot): \R^n\rightarrow \R^{n-1}$ be the projection that forgets the $i$th coordinate
%
and let ${\ve l}({\ve a}, {\ve c})=\pi_{\tau}({\ve c})-{c}_{\tau}{a}_{\tau}^{-1}\pi_{\tau}({\ve a})$.
Note that ${\ve l}$ corresponds to the dual slack.

Let $\rho_{d}$ denote the {\em covering constant} of the standard $d$-dimensional simplex, defined in Section \ref{geometry}.

\begin{thm}\hfill
\begin{itemize}
\item[(i)] Let ${\ve a}\in \Z^n_{>0}$ satisfy (\ref{nonzeroA}) and let ${\ve c}\in \Q^n$. Suppose that $({\ve a}, {\ve c})$ is generic. Then for $\tau=\tau({\ve a}, {\ve c})$ and ${\ve l}={\ve l}({\ve a}, {\ve c})$ we have
\be\begin{split}
\gap({\ve c}, {\ve a})\ge \rho_{{n-1}} ({a}_{\tau} l_1\cdots l_{n-1})^{1/{(n-1)}}-\|{\ve l}\|_1\,.
\label{optimal_bound}
\end{split}
\ee
\item[(ii)] For any $\epsilon>0$, there exists a vector ${\ve a}\in \Z^n_{>0}$, satisfying (\ref{nonzeroA}) and ${\ve c}\in \Q^n$ such that $({\ve a}, {\ve c})$ is generic and, in the notation of part (i), we have
\bea
\gap({\ve c}, {\ve a})< (\rho_{{n-1}}+\epsilon) ({a}_\tau l_1\cdots l_{n-1})^{1/{(n-1)}}-\|{\ve l}\|_1\,.
\eea
\end{itemize}
\label{thm_optimal_bound}
\end{thm}
The only known exact values of $\rho_{d}$ are $\rho_1=1$ and $\rho_2=\sqrt{3}$ (see  \cite{Fary}).
It was proved in \cite{AG}, that $\rho_{d}>(d!)^{1/d}$. 
For sufficiently large $d$ this bound is not far from being optimal. Indeed, $\rho_{d}\le (d!)^{1/d}(1+O(d^{-1}\log d))$ (see \cite{DF} and \cite{MS}).

Theorem \ref{thm_optimal_bound} is the main ingredient in the proof of the last theorem of this paper that shows that the value of $\epsilon$ in (\ref{average_ineq})
cannot be smaller than $1/(n-1)$.

\begin{thm}\label{lower_upper}
Fix $n\ge 3$. For $H$ large
\bea
\frac{1}{N(H)}\sum_{{\ve a}\in \NC(H)}\max_{{\ve c}\in \Q^n}\frac{\gap({\ve c},{\ve a})}{\|{\ve a}\|_{\infty}^{1/(n-1)}\|{\ve c}\|_1} \gg_n 1\,.
\eea
\end{thm}

\section{Discrete coverings and Frobenius numbers}\label{geometry}

For linearly independent ${\ve b}_1, \ldots, {\ve b}_k$ in $\R^d$, the set $\Lambda=\{\sum_{i=1}^{k} x_i {\ve b}_i,\, x_i\in \Z\}$ is a $k$-dimensional {\em lattice} with {\em basis} ${\ve b}_1, \ldots, {\ve b}_k$ and {\em determinant}
$\det(\Lambda)=(\det[{\ve b}_i\cdot {\ve b}_j]_{1\le i,j\le k})^{1/2}$, where ${\ve b}_i\cdot {\ve b}_j$ is the standard inner product of the basis vectors ${\ve b}_i$ and  ${\ve b}_j$. Recall that the {\em Minkowski sum} $X+Y$ of the sets $X, Y\subset \R^d$ consists of all points ${\ve x}+{\ve y}$ with ${\ve x}\in X$ and ${\ve y}\in Y$.
For a lattice $\Lambda\subset\R^d$ and ${\ve y}\in \R^d$, the set ${\ve y}+\Lambda$ is an {\em affine lattice} with determinant $\det(\Lambda)$.
For sets $K, S \subset \R^d$ and a lattice $\Lambda\subset \R^d$, the set $K + \Lambda$ is a  {\em covering}  of $S$ if $S \subset K + \Lambda$.

In what follows, $\mathcal{L}^d$ will denote the set of all $d$-dimensional lattices in $\R^d$. By $\mathcal{K}^d$ we will denote the set of all $d$-dimensional {\em convex bodies}, i.e., closed bounded convex sets with non-empty interior in $\R^d$.

\begin{lemma}\label{property_of_classical_coverings} Let $K\in\mathcal{K}^d$, $\Lambda\in\mathcal{L}^d$ and let
 $K + \Lambda$ be a covering of $\R^d$. Then for any vectors ${\ve x}, {\ve y}\in \R^d$,
we have $({\ve x}+K)\cap({\ve y}+\Lambda)\neq \emptyset$.
\end{lemma}

\begin{proof} It is sufficient to show that for any vector ${\ve x}\in \R^d$,
we have $({\ve x}+K)\cap \Lambda\neq \emptyset$.
Let ${\ve \lambda}$ be any point of $\Lambda$. Then ${\ve x}\in K+{\ve \lambda}$ if and only if $-{\ve \lambda}\in K+(-{\ve x})$.
Hence $\R^d$ is covered by the set $K + \Lambda$ if and only if for each vector ${\ve x}\in \R^d$,
the set ${\ve x}+K$ contains a point of $\Lambda$.
\qed
\end{proof}

For $K\in\mathcal{K}^d$ and $\Lambda\in\mathcal{L}^d$, we define  the {\em covering radius}  $\mu(K,\Lambda)$ as
\begin{equation*}
 \mu(K,\Lambda)  =\min\{\mu > 0 :  \mu K+ \Lambda \text{ is a covering of }\R^d\}\,.
\end{equation*}
For further results on covering radii in the context of the geometry of numbers see e.g. Gruber \cite{peterbible} and Gruber and Lekkerkerker \cite{GrLek}.

Let $S^d_{\ve 1}=\{{\ve x}\in \R^d_{\ge 0}: x_1+\cdots+x_d\le 1\}$ be the standard $d$-dimensional simplex.
The optimal lower bound in Theorem \ref{thm_optimal_bound} is expressed using the covering constant
$\rho_d=\rho_d(S^d_{\ve 1})$ defined as
\bea
\rho_d=\inf\{\mu(S^d_{\ve 1}, \Lambda): \det(\Lambda)=1 \}\,.
\eea

%
%

Let $\Lambda$ be a lattice in $\R^d$ with basis ${\ve b}_1, \ldots, {\ve b}_d$ and
let $\hat{\ve b}_i$ be the vectors obtained using the Gram-Schmidt orthogonalisation of ${\ve b}_1, \ldots, {\ve b}_d$:
\bea\begin{array}{l}
\hat{\ve b}_1= {\ve b}_1\,, \\
\hat{\ve b}_{i}= {\ve b}_i-\sum_{j=1}^{i-1}\mu_{i,j}\hat{\ve b}_j\,, \;\;j=2,\ldots,d\,,
\end{array}
\eea
where $\mu_{i,j}=({\ve b}_i\cdot\hat{\ve b}_j)/|\hat{\ve b}_j|^2$.

Define the 
box ${\hat B}={\hat B}({\ve b}_1, \ldots, {\ve b}_d)$ as
\bea
{\hat B}=[0,\hat{\ve b}_1)\times \cdots \times [0, \hat{\ve b}_d)\,.
\eea
We will need the following useful observation.
\begin{lemma}\label{coveringbox}
${\hat B}+\Lambda$ is a covering of $\R^d$.
\end{lemma}
A proof of Lemma \ref{coveringbox} is implicitly contained, for instance, in the proof of the classical result of Babai \cite{Babai} on the nearest lattice point problem (see Theorem 5.3.26 in \cite{GLS}). For completeness, we include a proof that follows along an argument of the proof of Theorem 5.3.26 in \cite{GLS}.
\begin{proof}
Let ${\ve x}$ be any point of $\R^d$. It is sufficient to find a point ${\ve y}\in \Lambda$ such that
\be\label{GS-box-representation}
{\ve x}-{\ve y}=\sum_{i=1}^d \lambda_i \hat{\ve b}_i\,,\;\; 0\le \lambda_i<1\,, \;\;1\le i \le d\,.
\ee
This can be achieved using the following procedure. First we write
\bea
{\ve x}=\sum_{i=1}^d \lambda_i^0 \hat{\ve b}_i\,.
\eea
Then we subtract $ \lfloor\lambda_d^0 \rfloor {\ve b}_d$ to get a representation
\bea
{\ve x}-\lfloor\lambda_d^0 \rfloor {\ve b}_d=\sum_{i=1}^d \lambda_i^1 \hat{\ve b}_i\,,
\eea
where $0\le \lambda_d^1<1$. Next subtract $ \lfloor\lambda_{d-1}^1 \rfloor {\ve b}_{d-1}$ and so on until we obtain the representation
(\ref{GS-box-representation}). The lemma is proved.
\qed
\end{proof}

Let now $\Lambda$ be a sublattice of $\Z^{d}$ of full rank and let $K\in\mathcal{K}^d$.
In the course of the proof of part (i) of Theorem \ref{thm_bounds_for_distance} we will need to work with coverings $K+\Lambda$ of $\Z^d$, that we refer to as {\em discrete coverings}. For this purpose, we will need the following auxiliary results.

By Theorem I (A) and Corollary 1 in Chapter I of Cassels \cite{Cassels}, there exists a unique basis ${\ve b}_1, \ldots, {\ve b}_d$ of the sublattice $\Lambda$
of the form
\be\label{special_basis}
\begin{array}{l}
{\ve b}_1= v_{11}{\ve e}_1\,,\\
{\ve b}_2= v_{21}{\ve e}_1+v_{22}{\ve e}_2\,,\\
\vdots\\
{\ve b}_d= v_{d1}{\ve e}_1+\cdots+v_{dd}{\ve e}_d\,,
\end{array}
\ee
where ${\ve e}_i$ are the standard basis vectors of $\Z^d$, the coefficients $v_{ij}$ are integers, $v_{ii}>0$ and $0\le v_{ij}<v_{jj}$.
Alternatively, the basis ${\ve b}_1, \ldots, {\ve b}_d$ can be obtained by taking the Hermite Normal Form of a basis matrix for $\Lambda$.

Define the box $B=B({\ve b}_1, \ldots, {\ve b}_d)$ as
\bea
B=[0,v_{11}-1]\times \cdots \times [0, v_{dd}-1]\,.
\eea

%
%
\begin{lemma}\label{discrete_covering_box}
$B+\Lambda$ is a covering of $\Z^d$.
\end{lemma}
\begin{proof}
Observe that for the basis (\ref{special_basis}) the box ${\hat B}={\hat B}({\ve b}_1, \ldots, {\ve b}_d)$ can be written as
\bea
{\hat B}=[0,v_{11})\times \cdots \times [0, v_{dd})\,.
\eea
The result now follows by Lemma \ref{coveringbox}.
\qed
\end{proof}


\begin{lemma}\label{integer_covering_simplex}
$(\det(\Lambda)-1)S^d_{\ve 1} + \Lambda$ covers $\Z^{d}$.
\end{lemma}

\begin{proof}
By Lemma \ref{discrete_covering_box}, it is sufficient to show that $B\subset (\det(\Lambda)-1)S^d_{\ve 1}$ or, equivalently,
\be\label{discrete_corner_in_simplex}
\sum_{i=1}^d (v_{ii}-1)\le \det(\Lambda)-1\,.
\ee
Noticing that $\det(\Lambda)=v_{11}\cdots v_{dd}$, the inequality (\ref{discrete_corner_in_simplex}) easily follows by induction for $d$.
%
%
%
%
%
%
%
%
\qed
\end{proof}

\begin{lemma}\label{property_of_coverings}
Suppose that $K + \Lambda$ is a covering of $\Z^d$. Then for any vectors ${\ve x}, {\ve y}\in \Z^d$,
we have $({\ve x}+K)\cap({\ve y}+\Lambda)\neq \emptyset$.
\end{lemma}
A proof of Lemma \ref{property_of_coverings} can be easily obtained from of the proof of Lemma \ref{property_of_classical_coverings}.



Given  $K\in\mathcal{K}^d$ and $\Lambda\in\mathcal{L}^d$, we define
the {\em discrete covering radius} $\mu(K,\Lambda; \Z^d)$ as
\begin{equation*}
\begin{split}
 \mu(K,\Lambda; \Z^d)  =\min\{\mu > 0 : \mu K+ \Lambda  \text{ is a covering of }\Z^d \}\,.
\end{split}
\end{equation*}

For  ${\ve y}=(y_1, \ldots, y_d)^T$ with nonzero entries we will denote $\sign(y_i)=y_i/|y_i|$. Let $\mathcal{O}_{\ve y}=\{\ve x\in\R^d\;|\;\sign(y_i)  x_i\ge 0,\;1\le i \le d\}$ be the orthant that contains the vector ${\ve y}$.    Next, for ${\ve a}\in \Z^n$ satisfying (\ref{nonzeroA}) we define the $(n-1)$-dimensional simplex
\begin{equation*}
S_{\ve a}= \left\{ {\ve x} \in \mathcal{O}_{\pi_n({\ve a})}: a_1\,x_1+\cdots +a_{n-1}\,x_{n-1}\leq 1 \right\}
\end{equation*}
and the $(n-1)$-dimensional lattice
\begin{equation*}
\Lambda_{\ve a} = \left\{ {\ve x}\in\Z^{n-1} : a_1\,x_1+\cdots + a_{n-1}\,x_{n-1}\equiv 0 \bmod |a_n| \right\}.
\end{equation*}

Given ${\ve a}=(a_1, \ldots, a_n)^T\in \Z_{>0}^n$ with $\gcd({\ve a})=1$,
the {\em Frobenius number} $\frob({\ve a})$ is least so that every integer $b>\frob({\ve a})$ can be represented as
$
b= a_1 x_1 +\cdots+ a_n x_n
$
with nonnegative integers $x_1,\ldots, x_n$.

Kannan \cite{Kannan} found the following very useful identities:
\begin{equation}\label{spaceKannan}
      \mu(S_{\ve a},\Lambda_{\ve a})= \frob({\ve a})+a_1+\cdots +a_n
\end{equation}
and
\begin{equation}\label{integralKannan}
       \mu(S_{\ve a},\Lambda_{\ve a}; \Z^{n-1})= \frob({\ve a})+a_n.
\end{equation}

\section{Proof of Theorem \ref{thm_bounds_for_distance}}

We will use the following notation. $\Lambda({\ve a},b)$ will denote the affine lattice formed by integer points in the affine hyperplane ${\ve a}^T{\ve x}=b$, that is
$\Lambda({\ve a},b)=\{{\ve x}\in \Z^n: {\ve a}^T{\ve x}=b\}$. 
We also set 
$Q({\ve a}, b)=\pi_n(P({\ve a}, b))$ and $L({\ve a},b)=\pi_n(\Lambda({\ve a},b))$.
Notice that the affine lattice $L(A,b)$ can be written in the form
\begin{equation}\label{lattice_as_congruence}
L({\ve a},b) = \left\{ {\ve x}\in\Z^{n-1} : a_1\,x_1+\cdots + a_{n-1}\,x_{n-1}\equiv b \bmod |a_n| \right\}\,.
\end{equation}
Furthermore,
 $L({\ve a},0)=\Lambda_{\ve a}$ is a lattice of determinant $\det(\Lambda_{\ve a})=a_n$ and
 $L({\ve a},b)=\Lambda_{\ve a}+{\ve y}$ for some ${\ve y}\in \Z^{n-1}$.

To prove part (i) we will start with two special cases. First we suppose that all entries of ${\ve a}$ are positive. In this setting, we obtain an upper bound for  $\idist({\ve a},b)$ in terms of the Frobenius number $\frob({\ve a})$. This bound will be also used in the proof of Theorem \ref{Ratio}.

\begin{lemma}\label{lemma_bounds_for_distance_via_Frobi}
Let ${\ve a}\in \Z_{>0}^{ n}$ satisfy (\ref{nonzeroA}) and  $b\in \Z$. Then 
\be\label{bound_via_Frobi}
\idist({\ve a},b)\le\frac{\frob({\ve a})+\|{\ve a}\|_{\infty}}{\min_i a_i}\,.
\ee

\end{lemma}

\begin{proof}

Then, if $b<0$ the polyhedron $P({\ve a},b)$ is empty and
in the case $b=0$ we have $P({\ve a},b)= {\ve 0}$.
Assume now that $b$ is a positive integer. Clearly, $P({\ve a}, b)$ is a simplex with vertices
\bea
\left(\frac{b}{a_1},0, \ldots, 0\right)^T, \left(0, \frac{b}{a_2}, \ldots, 0\right)^T, \ldots, \left(0, \ldots, 0, \frac{b}{a_n}\right)^T
\eea
and, consequently,
\be\label{cube}
P({\ve a}, b)\subset \left[0, \frac{b}{\min_i a_i}\right]^n\,.
\ee

Let ${\ve v}$ be any vertex of $P({\ve a},b)$. Rearranging the entries of ${\ve a}$, we may assume that
${\ve v}=(0,\ldots, 0, b/a_n)^T$.
If $b \le \mu(S_{\ve a},\Lambda_{\ve a}; \Z^{n-1})$ then (\ref{integralKannan}) combined with (\ref{cube}) implies (\ref{bound_via_Frobi}).
Suppose now that $b > \mu(S_{\ve a},\Lambda_{\ve a}; \Z^{n-1})$.
Then
\be\label{dilated_simplex_in_projection}
\mu(S_{\ve a},\Lambda_{\ve a}; \Z^{n-1}) S_{\ve a}\subset bS_{\ve a}= Q({\ve a}, b)\,.
\ee
%
By Lemma \ref{property_of_coverings}, applied to the covering $\mu(S_{\ve a},\Lambda_{\ve a}; \Z^{n-1}) S_{\ve a}+\Lambda_{\ve a}$ of $\Z^{n-1}$,  there is a point $(z_1, \ldots, z_{n-1})^T\in L({\ve a},b)\cap \mu(S_{\ve a},\Lambda_{\ve a}; \Z^{n-1}) S_{\ve a}$.
Hence, using (\ref{dilated_simplex_in_projection}) and the definition of the lattice $L({\ve a},b)$,
\bea
{\ve z}=\left(z_1, \ldots, z_{n-1}, \frac{b}{a_n}-\frac{a_1z_1+\cdots+a_{n-1}z_{n-1}}{a_n}\right)^T
\eea
is an integer point in the knapsack polyhedron $P({\ve a}, b)$.

Since $(z_1, \ldots, z_{n-1})^T\in \mu(S_{\ve a},\Lambda_{\ve a}; \Z^{n-1}) S_\A$, we have
\bea
||{\ve v}-{\ve z}||_\infty \le \frac{\mu(S_{\ve a},\Lambda_{\ve a}; \Z^{n-1})}{\min_i a_i}\le \frac{\frob({\ve a})+a_n}{\min_i a_i}\,,
\eea
where the last inequality follows from (\ref{integralKannan}). The lemma is proved.
\qed
\end{proof}

The next corollary will complete the proof of part (i) for vectors ${\ve a}$ with positive entries,

\begin{corollary}\label{coro_positive_entries}
Let ${\ve a}\in \Z^{n}_{>0}$ satisfy (\ref{nonzeroA}) and   $b\in \Z$. Then
\be\label{bound_obtained_via_Frobi}
\idist({\ve a},b)\le\|{\ve a}\|_{\infty}-1\,.
\ee
\end{corollary}

\begin{proof}
We use a classical upper bound for the Frobenius number due to Schur (see Brauer \cite{Brauer}):
\be\label{Schur}
\frob({\ve a})\le (\min_i a_i)\|{\ve a}\|_{\infty}-(\min_i a_i)-\|{\ve a}\|_{\infty}\,.
\ee
The bound (\ref{bound_via_Frobi}) combined with (\ref{Schur}) immediately implies (\ref{bound_obtained_via_Frobi}).

\end{proof}

Next, we will consider the case when  at least one of the entries of ${\ve a}$ is negative, the entries of ${\ve a}$ satisfy the condition
\be\label{main_assumption}
a_n= \min_{i=1,\ldots, n} |a_i|<\|{\ve a}\|_{\infty}
\ee
and the polyhedron $P(\pi_n({\ve a}), b)=\{{\ve x}\in \R^{n-1}_{\ge 0}: \pi_n({\ve a})^T {\ve x}=b\}\,$
is bounded or empty.

\begin{lemma}\label{lemma_bounds_for_distance_discrete_coverings}
Let ${\ve a}\in \Z^{n}$ satisfy (\ref{nonzeroA}) and   $b\in \Z$. If 
${\ve a}$ has at least one negative entry, (\ref{main_assumption}) holds and $P(\pi_n({\ve a}),b)$
is bounded or empty, then
\bea
\idist({\ve a},b)\le\|{\ve a}\|_{\infty}-1\,.
\eea

\end{lemma}

\begin{proof}

The vector ${\ve a}$ has at least one positive and at least one negative entry and, consequently, the polyhedron $P({\ve a}, b)$ is unbounded.
%
%
Since we assumed $a_n>0$, $P(\pi_n({\ve a}),b)$ can be bounded or empty only when all entries of $\pi_n({\ve a})$  are negative.
%

Suppose first that $b>0$, so that $P({\ve a},b)$ has the single vertex ${\ve v}=(0,\ldots, 0, b/a_n)^T$, the polyhedron $P(\pi_n({\ve a}),b)$ is empty and $Q({\ve a},b)=\R^{n-1}_{\ge 0}$.
By Lemma \ref{integer_covering_simplex}, $(a_n-1)S^{n-1}_{\ve 1}+\Lambda_{\ve a}$ covers $\Z^{n-1}$. The affine lattice $L({\ve a},b)$ is an integer translate of the lattice $\Lambda_{\ve a}$.
Hence, by Lemma \ref{property_of_coverings},  there is a point ${\ve y}\in L({\ve a},b)\cap(a_n-1)S^{n-1}_{\ve 1}$.
In view of (\ref{main_assumption}), we have $a_n\le \|{\ve a}\|_{\infty}-1$.

Hence
\be\label{first_cooredinates_zero_case}
\|{\ve y}\|_{\infty}<\|{\ve a}\|_{\infty}-1\,.
\ee

Let now
\bea
{\ve z}=\left(y_1, \ldots, y_{n-1}, \frac{b-a_1y_1-\cdots-a_{n-1}y_{n-1}}{a_n}\right)^T\,.
\eea
Since ${\ve y}\in  L({\ve a},b)\cap Q({\ve a},b)$, the point ${\ve z}$ is an integer point in the knapsack polyhedron $P({\ve a},b)$. Thus, in view of (\ref{first_cooredinates_zero_case}),
it is sufficient to check that $|z_n-b/a_n|\le  \|{\ve a}\|_{\infty}-1$. Since ${\ve y}\in (a_n-1)S^{n-1}_{\ve 1}$, we have $|a_1|y_1+\cdots+|a_{n-1}|y_{n-1}\le (a_n-1)\|{\ve a}\|_{\infty}$. Therefore
\bea\begin{split}
\left |z_n-\frac{b}{a_n}\right |\le \frac{(a_n-1)\|{\ve a}\|_{\infty}}{a_n}
< \|{\ve a}\|_{\infty}-1\,.
\end{split}
\eea

Suppose now that $b\le 0$ and choose any vertex ${\ve v}$ of the polyhedron $P({\ve a},b)$. We have ${\ve v}=(0,\ldots, 0, b/a_j, 0, \ldots, 0)^T$ for some $1\le j<n$ and, consequently,
${\ve w}=\pi_n({\ve v})$ is a vertex of the polyhedron $P(\pi_n({\ve a}),b)$.
Let ${\ve u}$ be the point obtained from ${\ve w}$
by rounding up its $j$th entry, that is
${\ve u}=(0,\ldots, 0, \lceil b/a_j\rceil, 0, \ldots, 0)^T$. Since all entries of $\pi_n({\ve a})$ are negative, we have  ${\ve u}\in Q({\ve a},b)$.

By Lemmas  \ref{integer_covering_simplex} and  \ref{property_of_coverings},
there is a point ${\ve y}\in L({\ve a},b)$ in the simplex ${\ve u}+(a_n-1)S^{n-1}_{\ve 1}$.
In view of (\ref{main_assumption}), we have $a_n\le \|{\ve a}\|_{\infty}-1$.
%
Hence
\be\label{w_minus_y}
\|{\ve w}-{\ve y}\|_{\infty}\le a_n-1+ \|{\ve u}-{\ve w}\|_{\infty} <\|{\ve a}\|_{\infty}-1\,.
\ee

Let now
\bea
{\ve z}=\left(y_1, \ldots, y_{n-1}, \frac{b-a_1y_1-\cdots-a_{n-1}y_{n-1}}{a_n}\right)^T\,.
\eea
Since ${\ve y}\in {\ve u}+\R^{n-1}_{\ge 0}\subset Q({\ve a},b)$, the point ${\ve z}$ is an integer point in the knapsack polyhedron $P({\ve a},b)$.
Thus, noticing (\ref{w_minus_y}), it is sufficient to check that $z_n\le  \|{\ve a}\|_{\infty}-1$.
Observe that $\pi_n({\ve a})^T{\ve w}=b$ and $\pi_n({\ve a})^T{\ve y}=a_1y_1+\cdots+a_{n-1}y_{n-1}$. 
Therefore,
\bea\label{rounding_up_bound}
\begin{split}
b-a_1y_1-\cdots-a_{n-1}y_{n-1}=\pi_n({\ve a})^T({\ve w}-{\ve y})=\pi_n({\ve a})^T({\ve w}-{\ve u}+({\ve u}-{\ve y}))\\ \le\lfloor |\pi_n({\ve a})^T({\ve u}-{\ve w})|\rfloor+(a_n-1)\|{\ve a}\|_{\infty}
\le |a_j|-1+ (a_n-1)\|{\ve a}\|_{\infty} \,.
\end{split}
\eea
The latter bound implies
\bea\begin{split}
z_n\le \left\lfloor\frac{ |a_j|-1+(a_n-1)\|{\ve a}\|_{\infty}}{a_n}\right\rfloor
\le \|{\ve a}\|_{\infty}-1\,.
\end{split}
\eea
The lemma is proved.
\qed
\end{proof}

Now, to prove the statement (i) of Theorem \ref{thm_bounds_for_distance} in the general case
we will proceed by induction on $n$.

The basis step $n=2$ is immediately settled by Corollary \ref{coro_positive_entries} and Lemma \ref{lemma_bounds_for_distance_discrete_coverings}.
%
%
%
Suppose now that $n\ge 3$ and the statement (i) of Theorem \ref{thm_bounds_for_distance} holds in all dimensions $2\le k<n$.
We may assume without loss of generality that the condition (\ref{main_assumption}) is satisfied. Indeed, rearranging the entries of ${\ve a}$ and replacing ${\ve a}$, $b$ by $-{\ve a}$, $-b$ we may assume that $
0<a_n= \min_{i=1,\ldots, n} |a_i|$.  Furthermore,  $\min_{i=1,\ldots, n} |a_i|=\|{\ve a}\|_{\infty}$ would imply that ${\ve a}=(\pm 1,\ldots, \pm 1, 1)^T$. In this case ${\ve a}$
is totally unimodular  and, consequently, $P({\ve a},b)$ is an integral polyhedron.

Furthermore, by Corollary \ref{coro_positive_entries} and Lemma \ref{lemma_bounds_for_distance_discrete_coverings}, we may assume that at least one of the entries of ${\ve a}$ is negative and the polyhedron $P(\pi_n({\ve a}),b)$ is unbounded.

Let ${\ve v}$ be any vertex of $P({\ve a},b)$. Observe that ${\ve v}$ has at most one nonzero entry $b/a_j$ for some $1\le j \le n$ and that ${\ve w}=\pi_n({\ve v})$ is a vertex of the polyhedron $Q({\ve a}, b)$.
Suppose first that ${\ve w}\neq {\ve 0}$. Rearranging the first $n-1$ entries of ${\ve a}$, we may assume without loss of generality that
${\ve w}=(0,\ldots,0,b/a_{n-1})^T$.

Clearly, ${\ve w}$ is a vertex of $P(\pi_n({\ve a}),b)$. Suppose that $P(\pi_n({\ve a}),b)\cap L({\ve a},b)$ is not empty.
Then, by the inductive hypothesis, there exists an integer point
${\ve y}=(y_1,\ldots, y_{n-1})^T\in P(\pi_n({\ve a}),b)$ such that $\|{\ve w}-{\ve y}\|_{\infty}\le \|\pi_n({\ve a})\|_{\infty}-1$. Hence the point ${\ve z}=(y_1,\ldots, y_{n-1},0)^T\in P({\ve a},b)$
satisfies (\ref{thm_upper}).

Next we will suppose that $P(\pi_n({\ve a}),b)\cap L({\ve a},b)=\emptyset$. Noting (\ref{lattice_as_congruence}),
we have
\be\label{all_points_in_the_lattice}
 P(\pi_n({\ve a}),t)\cap L({\ve a},t)=P(\pi_n({\ve a}),t)\cap \Z^{n-1}\;\;\text{ for any }t\in \Z\,.
\ee
Hence $P(\pi_n({\ve a}),b)\cap \Z^{n-1}=\emptyset$ and, taking into account that $P(\pi_n({\ve a}),b)$ is unbounded,  we have $h=\gcd(\pi_n({\ve a}))\ge 2$.
Since $\gcd(h, a_n)=1$ and $Q({\ve a},b)=\{{\ve x}\in\R_{\ge 0}^{n-1}: \pi_n({\ve a})^T{\ve x}\le b\}$, there exists an integer $t$ such that

\begin{itemize}

\item[(i)] $t$ is in the interval $[b-ha_n+1, b)$,

\item[(ii)]  $P(\pi_n({\ve a}),t)\cap L({\ve a},b)$ is not empty or, equivalently, $t\equiv 0\mod h$ and $t\equiv b\mod a_n$.

\item[(iii)]  $P(\pi_n({\ve a}),t)\subset Q({\ve a},b)$.

\end{itemize}
Notice that the condition (i) implies (iii).

Let us choose a vertex ${\ve p}$ of the polyhedron $P(\pi_n({\ve a}),t)$ in the following way. If ${\ve p}'=(0,\ldots, 0, t/a_{n-1})^T$ is a vertex of $P(\pi_n({\ve a}),t)$, then we set ${\ve p}={\ve p}'$.
Otherwise, we select ${\ve p}$ as an arbitrary vertex of $P(\pi_n({\ve a}),t)$. By the inductive assumption, there exists an integer point ${\ve y}=(y_1,\ldots, y_{n-1})^T\in P(\pi_n({\ve a}),t)$
such that
\bea
||{\ve p}-{\ve y}||_\infty \le  \frac{\|{\ve a}\|_{\infty}}{h}-1\,.
\eea
By (\ref{all_points_in_the_lattice}), we have ${\ve y}\in L({\ve a},b)$ and, using (iii), there should exist an integer point ${\ve z}=(y_1,\ldots, y_{n-1},z)^T\in P({\ve a},b)$.
Now $a_1y_1+\cdots+a_{n-1}y_{n-1}+a_n z=b$ implies $z=(b-t)/a_n$. Hence by (i) we have
\be
\label{last_entry}
|z|< h\le \|{\ve a}\|_{\infty}.
\ee

Recall that ${\ve p}=(0,\ldots, 0, t/a_j, 0, \ldots, 0)^T$ for some $1\le j\le n-1$ and  ${\ve w}=(0,\ldots,0,b/a_{n-1})^T$.
If $j=n-1$, then by (i)
\bea
||{\ve w}-{\ve y}||_\infty \le\left|\frac{b-t}{a_j}\right|+\frac{\|{\ve a}\|_{\infty}}{h}-1\le \left|\frac{ha_n-1}{a_j}\right|+\frac{\|{\ve a}\|_{\infty}}{h}-1\,.
\eea
On the other hand, if $j\neq n-1$ then, by construction of ${\ve p}$, we have  $t/a_{n-1}<0$, so that
$bt<0$. Consequently, by (i), we get $|b|+|t|\le ha_n-1$. Hence
\bea
||{\ve w}-{\ve y}||_\infty \le \max\left(\left|\frac{t}{a_j}\right|, \left|\frac{b}{a_{n-1}}\right|\right)+ \frac{\|{\ve a}\|_{\infty}}{h}-1\le \left|\frac{ha_n-1}{a_j}\right|+\frac{\|{\ve a}\|_{\infty}}{h}-1\,.
\eea

Taking into account (\ref{last_entry}), we have
\bea
||{\ve v}-{\ve z}||_\infty \le \max\left(\left|\frac{ha_n-1}{a_j}\right|+\frac{\|{\ve a}\|_{\infty}}{h}-1, \|{\ve a}\|_{\infty}-1\right)\,.
\eea
Suppose first that $\|{\ve a}\|_{\infty}=h$. Then $|a_1|=\cdots=|a_{n-1}|=h$ and, using the assumption  (\ref{main_assumption}),  we have $a_n\le h-1$. Now we get
\bea
\left|\frac{ha_n-1}{a_j}\right|+\frac{\|{\ve a}\|_{\infty}}{h}-1 \le \frac{h(h-1)-1}{h}+1-1< h-1\,.
\eea

We may now assume that $\|{\ve a}\|_{\infty}\ge 2h$. Then, using (\ref{main_assumption}),
\bea\begin{split}
 \frac{ha_n-1}{|a_j|}+\frac{\|{\ve a}\|_{\infty}}{h}-1\le h+\frac{\|{\ve a}\|_{\infty}}{h}-1
\le \|{\ve a}\|_{\infty}-1.
\end{split}
\eea

Let us now suppose that ${\ve w}={\ve 0}$, so that ${\ve v}=(0,\ldots, 0, b/a_n)^T$.
In this setting, we will need to consider separately the case $\|{\ve a}\|_{\infty}=h=\gcd(\pi_n({\ve a}))$.
There exists an index $1\le i< n$ such that $\pi_i({\ve a})$ has at least one negative entry. Hence, the polyhedron $P(\pi_i({\ve a}),b)$ is unbounded.
Since  $|a_1|=\cdots=|a_{n-1}|=h$, we have $\gcd(\pi_i({\ve a}))=1$.

Next,
$\pi_i({\ve v})$ is a vertex of $P(\pi_i({\ve a}),b)$ and, since $\gcd(\pi_i({\ve a}))=1$, we have  $P(\pi_i({\ve a}),b)\cap \Z^{n-1}\neq \emptyset$.
Then, by the inductive hypothesis, there exists an integer point
${\ve y}\in P(\pi_i({\ve a}),b)$ such that $\|\pi_i({\ve v})-{\ve y}\|_{\infty}\le \|\pi_i({\ve a})\|_{\infty}-1$.  Therefore, (\ref{thm_upper}) holds with the integer point ${\ve z}=(y_1,\ldots, y_{i-1},0, y_{i+1}, \ldots, y_{n-1})^T\in P({\ve a},b)$.  For the rest of the proof of the part (i) we will assume without loss of generality that
\be\label{exclude_gcd}
\|{\ve a}\|_{\infty}\ge 2h\,.
\ee


Since ${\ve 0}={\ve w}\in Q({\ve a},b)$, there exists an integer $t$ such that

\begin{itemize}

\item[(i)] $t$ is in the interval $[-ha_n+1, 0]$,

\item[(ii)]  $P(\pi_n({\ve a}),t)\cap L({\ve a},b)$ is not empty,

\item[(iii)]  $P(\pi_n({\ve a}),t)\subset Q({\ve a},b)$.

\end{itemize}

Let ${\ve p}$ be a vertex of the polyhedron $P(\pi_n({\ve a}),t)$. By the inductive assumption, there exists an integer point ${\ve y}=(y_1,\ldots, y_{n-1})^T\in P(\pi_n({\ve a}),t)$
such that
\bea
||{\ve p}-{\ve y}||_\infty \le  \frac{\|{\ve a}\|_{\infty}}{h}-1\,.
\eea
By (\ref{all_points_in_the_lattice}), we have ${\ve y}\in L({\ve a},b)$. Therefore, using (iii), there exists an integer point ${\ve z}=(y_1,\ldots, y_{n-1},z)^T\in P({\ve a},b)$.
Next, $a_1y_1+\cdots+a_{n-1}y_{n-1}+a_n z=b$ implies $z=(b-t)/a_n$ and, by (i), we have
\be
\label{last_entry_for_zerocase}
\left |z-\frac{b}{a_n}\right |\le h-\frac{1}{a_n}<  \|{\ve a}\|_{\infty}-1\,,
\ee
where the last inequality follows from ({\ref{exclude_gcd}}).

Observe that ${\ve p}=(0,\ldots, 0, t/a_j, 0, \ldots, 0)^T$ for some $1\le j\le n-1$. Since ${\ve v}=(0,\ldots,0,b/a_{n})^T$ and taking into account (\ref{last_entry_for_zerocase}), we have
\bea
||{\ve v}-{\ve z}||_\infty \le \max\left(\left|\frac{t}{a_j}\right|+\frac{\|{\ve a}\|_{\infty}}{h}-1, \|{\ve a}\|_{\infty}-1\right)\,.
\eea
Now, using (i), (\ref{main_assumption}) and (\ref{exclude_gcd}),
\bea\begin{split}
\left|\frac{t}{a_j}\right|+\frac{\|{\ve a}\|_{\infty}}{h}-1 \le \frac{ha_n-1}{|a_j|}+\frac{\|{\ve a}\|_{\infty}}{h}-1\le h+\frac{\|{\ve a}\|_{\infty}}{h}-1\\
\le \|{\ve a}\|_{\infty}-1.
\end{split}
\eea

This completes the proof of part (i).

To prove part (ii), we set ${\ve a}=(k,\ldots,k,1)^T$ and $b=k-1$. The knapsack polyhedron $P({\ve a},b)$ contains precisely one integer point,  ${\ve z}=(k-1)\cdot\ve{e}_n$,
where $\ve{e}_i$ denotes the $i$-th unit-vector. Choosing the vertex ${\ve v}=\frac{k-1}{k}\cdot\ve{e}_1$ of $P({\ve a},b)$ we get
$\|{\ve v}-{\ve z}\|_{\infty}=k-1 =\|{\ve a}\|_{\infty}-1$.

\section{Proof of Corollary \ref{coro_upper_bound}}

The part (i) immediately follows from part (i) of Theorem \ref{thm_bounds_for_distance}.
To prove part (ii) it is sufficient to consider the same ${\ve a}, b$ as in the proof of part (ii) of Theorem \ref{thm_bounds_for_distance} and take $\ve{c}=\ve{e}_n$. Then the integer programming problem~\eqref{initial_IP} has precisely one feasible, and therefore optimal, integer solution $(k-1)\cdot\ve{e}_n$. Thus $IP(\ve{c},{\ve a},b)=k-1$.
The corresponding linear relaxation~\eqref{initial_LP} has the, in general not unique, optimal solution $\frac{k-1}{k}\cdot\ve{e}_1$ with $LP(\ve{c},{\ve a},b)=0$.
Hence, $\gap(\ve{c},{\ve a}) \ge IG(\ve{c}, {\ve a},b) = k-1 = (\|{\ve a}\|_{\infty}-1) \|{\ve c}\|_1$.

\section{Proof of Theorem \ref{Ratio}}

We will first obtain an analog of Lemma \ref{lemma_bounds_for_distance_via_Frobi} for the unbounded polyhedra $P({\ve a},b)$.

Set  ${\ve a}^+=(|a_1|, \ldots, |a_n|)^T$. For convenience, we will work with the quantity
\bea
\ffrob({\ve a}^+)= \frob({\ve a}^+)+|a_1| +\cdots+ |a_n|\,.
\eea
\begin{lemma}\label{lemma_bounds_for_distance_via_Frobi_unbounded_polyhedra}
Let ${\ve a}\in \Z^{ n}$ satisfy (\ref{nonzeroA}) and  $b\in \Z$. If $P({\ve a},b)$ is unbounded then 
\bea
\idist({\ve a},b)\le\frac{(n-1)\ffrob({\ve a}^+)}{\min_i |a_i|}\,.
\eea

\end{lemma}

\begin{proof}
We will use the notation from the proof of Theorem \ref{thm_bounds_for_distance}. Let ${\ve v}$ be any vertex of $P({\ve a},b)$.
Rearranging the entries of ${\ve a}$ and replacing ${\ve a}$, $b$ by $-{\ve a}$, $-b$, we may assume that
${\ve v}=(0,\ldots, 0, b/a_n)^T$ and $a_n>0$. The unbounded polyhedron
\bea
Q({\ve a},b)=\pi_n(P({\ve a},b))=\{{\ve x}\in \R^{n-1}_{\ge 0}: a_1 x_1 +\cdots+ a_{n-1} x_{n-1}\le b\}\,
\eea
can be represented in the form
$Q({\ve a},b)=R({\ve a},b)+C({\ve a},b)$, where $R({\ve a},b)$ is a polytope and
\bea
C({\ve a},b)=\{{\ve x}\in \R^{n-1}_{\ge 0}: a_1 x_1 +\cdots+ a_{n-1} x_{n-1}\le 0\}\,
\eea
is the recession cone of the polyhedron $Q({\ve a},b)$ (see e. g. \cite[Section 8.2]{Schrijver}).
For $t\ge 0$ let  ${\ve u}(t)=(u_1(t), \ldots, u_{n-1}(t))^T$
with
\bea
u_i(t)=\left\{
\begin{array}{ll}
-t/a_i, & \text{ if }a_i<0\,,\\
0, & \text{ otherwise}\,.
\end{array}
\right.
\eea

Observe that ${\ve u(t)}+tS_{\ve a}\subset C({\ve a},b)$. Indeed, it is easy to check that $C({\ve a},b)$ contains all vertices
${\ve u}(t)$, ${\ve u}(t)+ (t/a_i) {\ve e}_i$, $1\le i\le n-1$ of the simplex ${\ve u(t)}+tS_{\ve a}$.

Recall that $\sign(a_i)=a_i/|a_i|$ for $1\le i \le n$.
Let $\phi$ be a linear map that sends a vector ${\ve x}=(x_1, \ldots, x_{n-1})^T\in \R^{n-1}$ to the vector $\phi({\ve x})= (\sign(a_1)x_1, \ldots, \sign(a_{n-1})x_{n-1})^T$. Observe that $\phi(S_{{\ve a}^+})=S_{\ve a}$ and $\phi(\Lambda_{{\ve a}^+})=\Lambda_{\ve a}$ . Therefore, by  (\ref{spaceKannan}) and the linear invariance of the covering radius, we have
\be\label{positive_Frobi_covering}
\ffrob({\ve a}^+)= \mu(S_{{\ve a}^+},\Lambda_{{\ve a}^+}) =  \mu(S_{\ve a}, \Lambda_{\ve a}).
\ee

Let $t_0=\ffrob({\ve a}^+)$.
Hence, using  (\ref{positive_Frobi_covering}), $t_0 S_{\ve a} + \Lambda_{\ve a}$ is a covering of $\R^{n-1}$.
We have  ${\ve u}(t_0)+t_0 S_{\ve a}\subset C({\ve a},b)\subset Q({\ve a},b)$. By Lemma \ref{property_of_classical_coverings}, there exists a point ${\ve y}$ of the affine lattice
$L({\ve a},b)$ in ${\ve u}(t_0)+t_0 S_{\ve a}$.
Therefore, $|y_i|\le t_0/|a_i|$ for $1\le i \le n-1$.
Hence, there exists a point ${\ve z}=(y_1, \ldots, y_{n-1},z)^T\in \Lambda({\ve a},b)\cap P({\ve a},b)$ with
\bea\begin{split}
z=\frac{b-a_1y_1-\cdots -a_{n-1}y_{n-1}}{a_n}=\frac{b}{a_n}-\frac{a_1y_1+\cdots+a_{n-1}y_{n-1}}{a_n}
\le v_n+\frac{(n-1)t_0}{a_n}\,.
\end{split}
\eea
The lemma is proved.
\qed
\end{proof}

Set
\bea
\RR= \{{\ve a}\in \Z^{ n}: 0<a_1\le\cdots \le a_n\} \,
\eea
and recall that
\bea
 N_{\epsilon}(t,H)=\#\left\{{\ve a}\in {\NC}(H): \max_{b\in\Z}\frac{\idist({\ve a},b)}{\|{\ve a}\|_{\infty}^{\epsilon}}>t \right\}\,.
\eea
%

%
%
By Lemmas \ref{lemma_bounds_for_distance_via_Frobi} and \ref{lemma_bounds_for_distance_via_Frobi_unbounded_polyhedra}, we have
%
%
\be\label{basic_bound}
\begin{split}
 N_{\epsilon}(t,H)\ll_n\, \#\left\{{\ve a}\in {\NC}(H)\cap \RR: \frac{(n-1)\ffrob({\ve a})}{a_1 a_n^{\epsilon}}>t \right\} \,.
\end{split}
\ee

We may assume $t \ge 10$ since otherwise (\ref{main_bound})
follows from $N_{\epsilon}(t,H)/N(H)\le 1$.
%
We keep $t' \in [1, t]$, to be fixed later. Then, setting $s({\ve a})= a_{n-1} a_n^{1/(n-1)}$ and noting (\ref{basic_bound}), we get
\be\label{two_terms}
\begin{split}
 N_{\epsilon}(t,H)\ll_n\, \#\left\{{\ve a}\in {\NC}(H)\cap \RR: \frac{(n-1)\ffrob({\ve a})}{s({\ve a})}>t'  \mbox{ or } \frac{s({\ve a})}{a_1a_n^{\epsilon}} > \frac{t}{t'}\right\}
\\
\le \,\#\left\{{\ve a}\in {\NC}(H)\cap \RR: \frac{\ffrob({\ve a})}{s({\ve a})}>\frac{t'}{n-1}\right\} \\ + \,\,\#\left\{{\ve a}\in {\NC}(H)\cap \RR: \frac{a_{n-1}}{a_1 a_n^{\epsilon-1/(n-1)}} > \frac{t}{t'}\right\}\,.
\end{split}
\ee

The first of the last two terms in (\ref{two_terms}) can be estimated using a special case of Theorem 3 in Str\"ombergsson \cite{Str}.

\begin{lemma}\label{first_term_bound}

\be\label{Str_bound}
\#\left\{{\ve a}\in {\NC}(H)\cap\RR: \frac{\ffrob({\ve a})}{s({\ve a})}>\frac{r}{n-1}\right\}\ll_n \frac{1}{r^{n-1}} N(H) \,,
\ee
uniformly over all $r>0$ and $H\ge 1$.
\end{lemma}
\begin{proof}
The inequality (\ref{Str_bound}) immediately follows from Theorem 3 in \cite{Str} applied with
${\mathcal D}=[0,1]^{n-1}$.
\qed\end{proof}
To estimate the last term in (\ref{two_terms}), we will need the following lemma.

\begin{lemma}\label{second_term_bound}

\be\label{second_term_ineq}
\#\left\{{\ve a}\in {\NC}(H)\cap \RR: \frac{a_{n-1}}{a_1 a_n^{\epsilon-1/(n-1)}} > r\right\}\ll_n \frac{1}{rH^{\epsilon-1/(n-1)}}N(H) \,,
\ee
uniformly over all $r>0$ and $H\ge 1$.
\end{lemma}

\begin{proof}

Since ${\ve a}\in \RR$, we have $a_{n-1}\le a_n$. Hence
\bea
\#\left\{{\ve a}\in {\NC}(H)\cap \RR: \frac{a_{n-1}}{a_1 a_n^{\epsilon-1/(n-1)}} > r\right\}
\le  \#\left\{{\ve a}\in {\NC}(H)\cap \RR: a_n^{1+1/(n-1)-\epsilon}> r a_1\right\}\,.
\eea
Furthermore, all ${\ve a}\in {\NC}(H)\cap \RR$ with $a_n^{1+1/(n-1)-\epsilon}> r a_1$ are in the set
\bea
U=\{{\ve a}\in \Z^{ n}: 0< a_1< H^{1+1/(n-1)-\epsilon}/r, 0< a_i\le H, i=2, \ldots, n\}\,.
\eea
Since $\#(U\cap\Z^n)< \min\{H^{n+1/(n-1)-\epsilon}/r, H^n\}$ and $N(H)\asymp_n H^n$ (see e.g. Theorem 1 in \cite{SchmidtDuke}), the result follows.
\qed
\end{proof}

%


Then by (\ref{two_terms}), (\ref{Str_bound}) and (\ref{second_term_ineq})

\be\label{two_terms2}
 \frac{N_{\epsilon}(t,H)}{N(H)}
\ll_n \frac{1}{(t')^{n-1}} + \frac{t'}{tH^{\epsilon-1/(n-1)}}\,.
\ee

Next, we will bound $H$ from below in terms of $t$, similar to Theorem 3 in \cite{Str}.  The upper bound of Schur (\ref{Schur})  implies
$\ffrob({\ve a})< n a_1 a_n$.
Thus, using (\ref{basic_bound}),
\bea
\begin{split}
 N_{\epsilon}(t,H)\ll_n  \#\left\{{\ve a}\in {\NC}(H)\cap \RR: \frac{(n-1)\ffrob({\ve a})}{a_1 a_n^{\epsilon}}>t \right\} \\ \le \#\left\{{\ve a}\in {\NC}(H)\cap \RR: a_n^{1-\epsilon}>\frac{t}{n(n-1)} \right\}\,.
\end{split}
\eea

The latter set is empty if $H \le (t/(n(n-1)))^{\frac{1}{1-\epsilon}}$. Hence we may assume
\be\label{Tviat}
H > \left(\frac{t}{n(n-1)}\right)^{\frac{1}{1-\epsilon}}\,.
\ee

Using (\ref{two_terms2}) and (\ref{Tviat}), we have
\be\label{viaalpha}
\frac{N_{\epsilon}(t,H)}{N(H)}\ll_n  \frac{1}{(t')^{n-1}} + \frac{t'}{  t^{ 1+\frac{ 1 }{ 1-\epsilon } \left(\epsilon -\frac{1}{n-1}\right) }  }\,.
\ee

To minimise the exponent of the right hand side of (\ref{viaalpha}), set  $t'=t^\beta$ and choose $\beta$ with
\be\label{equal_powers}
\beta (n-1)= 1+\frac{ 1 }{ 1-\epsilon } \left(\epsilon -\frac{1}{n-1}\right)-\beta\,.
\ee
We get
\bea
\beta= \frac{n -2}{n(n-1)(1-\epsilon)}
\eea
and, by (\ref{viaalpha}) and (\ref{equal_powers}),
\bea
 \frac{N_{\epsilon}(t,H)}{N(H)} \ll_n t^{-\alpha(\epsilon,n)}\,
\eea
with $\alpha(\epsilon,n)=\beta (n-1)$. The theorem is proved.

\section{Proof of Corollary \ref{average}}
It is sufficient to show (\ref{average_ineq}) for
\be\label{epsilon_small}
\frac{2}{n}<\epsilon<\frac{3}{4}\,.
\ee
Observe that
 the conditions $n\ge 3$ and $\epsilon>2/n$ imply that in (\ref{main_bound}) $\alpha(\epsilon, n)>1$.
For integers $s$ consider vectors ${\ve a}\in \NC(H)$ with
\be\label{exponents}
e^{s-1}\le \max_{{\ve c}\in \Q^n}\frac{\gap({\ve c}, {\ve a})}{\|{\ve a}\|_{\infty}^{\epsilon}\|{\ve c}\|_1}<e^s\,.
\ee
By (\ref{gap-distance}), we have
\bea
\max_{{\ve c}\in \Q^n}\frac{\gap({\ve c}, {\ve a})}{\|{\ve a}\|_{\infty}^{\epsilon}\|{\ve c}\|_1}\le \max_{b\in\Z}\frac{\idist({\ve a},b)}{\|{\ve a}\|_{\infty}^{\epsilon}}\,.
\eea
Therefore, the contribution of vectors satisfying (\ref{exponents}) to the sum
\bea
\sum_{{\ve a}\in \NC(H)}\max_{{\ve c}\in \Q^n}\frac{\gap({\ve c}, {\ve a})}{\|{\ve a}\|_{\infty}^{\epsilon}\|{\ve c}\|_1}
\eea
is
\bea
\le N_{\epsilon}(e^{s-1}, H)e^{s}\ll_n e^{-\alpha(\epsilon, n) s}e^s N(H)\,,
\eea
where the last inequality holds by (\ref{main_bound}) and the upper bound in (\ref{epsilon_small}).
Therefore
\bea
\frac{1}{N(H)}\sum_{{\ve a}\in \NC(H)}\max_{{\ve c}\in \Q^n}\frac{\gap({\ve c}, {\ve a})}{\|{\ve a}\|_{\infty}^{\epsilon}\|{\ve c}\|_1}\ll_n \sum_{s=1}^\infty e^{s(1-\alpha(\epsilon, n))}\,.
\eea
Finally, observe that the series
\bea
\sum_{s=1}^\infty e^{s(1-\alpha(\epsilon, n))}
\eea
is convergent for $\alpha(\epsilon, n)>1$.

\section{Proof of Theorem \ref{thm_optimal_bound}}

We will first show that $\gap({\ve c},{\ve a})$ is bounded from below by the lattice programming gap
associated with a certain lattice program.

For a  vector $\bar {\ve l}\in \Q^{n-1}_{>0}$, a $(n-1)$-dimensional lattice $\Lambda\subset\Z^{n-1}$ and ${\ve r}\in \Z^{n-1}$
consider the lattice program (also referred to as the {\em group problem})
\be\begin{split}
\min\{\bar {\ve l}^T {\ve x}: {\ve x} \equiv {\ve r} (\modulo \Lambda), {\ve x}\in \R^{n-1}_{\ge 0}\}\,.
\end{split}
\label{generic_group_relaxation}
\ee
Here ${\ve x} \equiv {\ve r} (\modulo \Lambda)$ if and only if ${\ve x} - {\ve r}$ is a point of $\Lambda$.

Let $m(\Lambda,\bar{\ve l}, {\ve r})$ denote the value of the minimum in (\ref{generic_group_relaxation}).
The {\em lattice programming gap} $\gap(\Lambda, \bar{\ve l})$ of (\ref{generic_group_relaxation}) is defined as
\be\begin{split}
\gap(\Lambda,\bar{\ve l})=\max_{{\ve r}\in \Z^{n-1}}m(\Lambda,\bar{\ve l}, {\ve r})\,.
\end{split}
\label{maximum_gap}
\ee
The lattice programming gaps were introduced and studied for sublattices of all dimensions in $\Z^{n-1}$ by Ho\c{s}ten and Sturmfels \cite{HS}.

To proceed with the proof of the part (i), we assume without loss of generality that $\tau({\ve a}, {\ve c})=\{n\}$.
The  {\em (Gomory's) group relaxation} to (\ref{initial_IP}) is a lattice program
\be\label{technical}\begin{split}
\min\{ {\ve l}^T {\ve x}: {\ve x} \equiv {\ve r}\; (\modulo \Lambda_{{\ve a}}), {\ve x}\in \R^{n-1}_{\ge 0}\}\,,
\end{split}
\ee
where ${\ve l}={\ve l}({{\ve a}},{\ve c})$ and ${\ve r}\in \Z^{n-1}$ is any point of the affine lattice $\pi_n(\Lambda({\ve a}, b))$.  We refer the reader to \cite[Section 24.2]{Schrijver} and
 \cite{RT_structure} for a detailed introduction to the theory of group relaxations.

%
%

The group relaxation (\ref{technical}) provides a lower bound for the integrality gap of (\ref{initial_IP}). Specifically, we have
\bea
IG({\ve c}, {\ve a}, b)\ge m(\Lambda_{{\ve a}},{\ve l}, {\ve r})
\eea
and, consequently,
\be\label{IPGbiggerLPG}
\gap({\ve c}, {\ve a})\ge \gap(\Lambda_{\ve a}, {\ve l})\,.
\ee
We will need the following result, obtained in \cite{lpg}.
\begin{proposition}[Theorem 1.2 (i) in \cite{lpg}]
For any $\bar {\ve l}=(\bar l_1, \ldots, \bar l_k)^T\in \Q^k_{>0}$, $k\ge 2$, and any $k$-dimensional lattice $\Lambda\subset \Z^k$
\bea
\gap(\Lambda,\bar {\ve l})\ge \rho_k (\det(\Lambda) \bar l_1 \cdots \bar l_k)^{1/k}-\|\bar {\ve l}\|_1,.
\label{optimal_bound_lpg}
\eea
%
\label{optimal_bound_lpg}
\end{proposition}

Note that for $n=2$ we have $\gap(\Lambda_{\ve a}, {\ve l})=l_1(|a_2|-1)$ and thus (\ref{IPGbiggerLPG}) implies (\ref{optimal_bound}).
For $n>2$, the bound (\ref{optimal_bound}) immediately follows from (\ref{IPGbiggerLPG}) and Proposition \ref{optimal_bound_lpg}.

The proof of the part (ii) will be based on the following lemma.

\begin{lemma}\label{Lpg_link} Let ${\ve a}\in \Z^n_{>0}$ satisfy (\ref{nonzeroA}),   ${\ve c}=(a_1, \ldots, a_{n-1}, 0)^T$ and ${\ve l}=(a_1, \ldots, a_{n-1})^T$. Then
\be\label{Lpg_via_gap}
\gap({\ve c}, {\ve a})= \gap(\Lambda_{\ve a}, {\ve l})\,.
\ee
\end{lemma}

\begin{proof}
Recall that $\Lambda({\ve a},b)=\{{\ve x}\in \Z^n: {\ve a}^T{\ve x}=b\}$ denotes the affined lattice formed by  integer points in the affine hyperplane ${\ve a}^T{\ve x}=b$ and $P({\ve a},b)=\{{\ve x}\in\R_{\ge 0}: {\ve a}^T{\ve x}=b\}$ denotes the knapsack polytope.
The assumption ${\ve a}\in \Z^n_{>0}$ implies that the linear programming relaxation (\ref{initial_LP}) is feasible if and only if $b$ is nonnegative. Suppose that for a nonnegative $b$ the knapsack problem (\ref{initial_IP}) has solution ${\ve y}\in \Z^n_{\ge 0}$. Then  for ${\ve r}=\pi_n({\ve y})\in \Z^{n-1}_{\ge 0}$
\bea
\pi_n(\Lambda({\ve a},b))={\ve r}+\Lambda_{\ve a}\,.
\eea
As $c_n=0$, the optimal value of the linear programming relaxation $LP({\ve c},{\ve a}, b)=0$.
Therefore, noting that ${\ve c}=(a_1, \ldots, a_{n-1}, 0)^T$ and ${\ve l}=\pi_n({\ve c})$,
\be\label{Extra_constraint}
IG({\ve c},{\ve a}, b)=\min\{{\ve l}^T {\ve x}: {\ve x}\in {\ve r}+\Lambda_{\ve a}\,, {\ve x}\in \pi_n(P({\ve a}, b)) \}\,.
\ee
Since
\bea
\pi_n(P({{\ve a}}, b))=bS_{{\ve a}}=\{{\ve x}\in \R^{n-1}_{\ge 0}: {\ve l}^T {\ve x}\le b\}\,
\eea
and ${\ve l}^T {\ve r}\le {\ve a}^T{\ve y}=b$, the constraint ${\ve x}\in \pi_n(P({{\ve a}}, b))$ in (\ref{Extra_constraint}) can be replaced by ${\ve x}\in \R^{n-1}_{\ge 0}$. Consequently, we have
\bea 
IG({\ve c},{\ve a}, b)= m(\Lambda_{\ve a},{\ve l}, {\ve r})\,.
\eea
Hence, by (\ref{maximum_gap}), we obtain
\be\label{gap_less}
\gap({\ve c}, {\ve a})\le  \gap(\Lambda_{\ve a}, {\ve l})\,.
\ee
Suppose now that $\gap(\Lambda_{\ve a}, {\ve l})= m(\Lambda_{\ve a},{\ve l}, {\ve r}_0)$. Then
\bea
IG({\ve c}, {\ve a}, \pi_n({\ve a})^T{\ve r}_0)= m(\Lambda,{\ve l}, {\ve r}_0)\,.
\eea
 Together with (\ref{gap_less}), this implies (\ref{Lpg_via_gap}).
\qed\end{proof}

As was shown in the proof of Theorem 1.1 in \cite{lpg}, for ${\ve l}=(a_1, \ldots, a_{n-1})^T$
\bea
\gap(\Lambda_{\ve a}, {\ve l})=\frob({\ve a})+a_n\,.
\eea
Thus we obtain the following corollary.

\begin{cor}\label{Frobi_link} Let ${\ve a}=(a_1, \ldots, a_n)^T\in \Z^n_{>0}$ satisfy (\ref{nonzeroA}) and ${\ve c}=(a_1, \ldots, a_{n-1}, 0)^T$. Then
\bea
\gap({\ve c}, {\ve a})= \frob({\ve a})+a_n\,.
\eea
\end{cor}

To complete the proof of Theorem \ref{thm_optimal_bound} we will need the following result, obtained in \cite{AG}.
\begin{proposition}[see Theorem 1.1 (ii) in \cite{AG}]
For any $\epsilon>0$, there exists a vector ${\ve a}\in \Z^{n}_{>0}$  such that
\bea
\frob({\ve a})<(\rho_{n-1}+\epsilon) (a_1\cdots a_n)^{1/(n-1)}- \|{\ve a}\|_1\,.
\eea
\label{optimal_bound_Frobenius}
\end{proposition}

For $n=2$, we have
\bea
g({\ve a})=a_1a_2-a_1-a_2
\eea
 by a classical result of Sylvester (see e.g. \cite{Alf}). Hence
the part (ii) immediately follows from Corollary \ref{Frobi_link}.
For $n>2$, 
the part (ii) follows from  Corollary \ref{Frobi_link} and Proposition \ref{optimal_bound_Frobenius}.

\section{Proof of Theorem \ref{lower_upper}}

We will
denote for ${\ve a}\in Q(H)\cap\Z^n_{>0}$ the index of a maximum coordinate by
$i({\ve a})$ and we set ${\ve c}_{{\ve a}}=-\ve e_{i({\ve a})}$.  
The tuples $({\ve a},{\ve c}_{\ve a} )$ are generic and in view of Theorem
\ref{thm_optimal_bound} we find
\begin{equation*}
\begin{split}
\gap({{\ve c}_{\ve a}}, {\ve a})  & \geq \rho_{n-1} a_{i({\ve a})}^{1/(n-1)}\left(\prod_{i=1, i\ne
  i({\ve a})}^{n}\frac{a_i}{a_{i({\ve a})}}\right)^{1/(n-1)} - \sum_{i=1, i\ne i({\ve a})}^n
\frac{a_i}{a_{i({\ve a})}}  \\
& \geq  \frac{1}{\|\ve {\ve a}\|_\infty}\rho_{n-1} \left(\prod_{i=1}^{n} a_i\right)^{1/(n-1)}-n+1.
\end{split}
\end{equation*}
Hence
\begin{equation*}
\frac{\gap({\ve c}_{\ve a},{\ve a})}{\|{\ve a}\|_\infty^{1/(n-1)}} \geq
\rho_{n-1} \frac{1}{\|{\ve a}\|_\infty^{1+1/(n-1)}}
\left(\prod_{i=1}^{n} a_i\right)^{1/(n-1)}-\frac{n-1}{\|{\ve a}\|_\infty^{1/(n-1)}}.
\end{equation*}
Next we observe that
\begin{equation*}
\begin{split}
\sum_{{\ve a}\in Q(H)\cap\Z^n_{>0}} \frac{1}{\|{\ve a}\|_\infty^{1/(n-1)}}& \leq
n\,H^{n-1}\sum_{t=1}^H  \frac{1}{t^{1/(n-1)}} \leq
n\,H^{n-1}\left(1+\int_{1}^H \frac{1}{t^{1/(n-1)}}{\rm d}t\right)\\
&\leq \frac{n-1}{n-2}n\,H^{n-1} H^{1-1/(n-1)} \leq 2\,n\, H^{n-1/(n-1)}
\end{split}
\end{equation*}
for $n\geq 3$.
Since $N(H)\asymp_n H^n$ (see e.g. Theorem 1 in \cite{SchmidtDuke}), for sufficiently large $H$ we have $N(H)\geq (1/3)H^n$, say. Thus, so far we know that
\begin{equation*}
\begin{split}
\frac{1}{N(H)}\sum_{{\ve a}\in Q(H)}  \max_{c\in\Q^n}\frac{\gap({\ve c}, {\ve a})}{\|{\ve a}\|_\infty^{1/(n-1)}\|\ve c\|_1}  \geq \frac{1}{N(H)}
\sum_{{\ve a}\in Q(H)\cap\Z^n_{>0}}
\frac{\gap({\ve c}_{\ve a}, {\ve a})}{\|{\ve a}\|_\infty^{1/(n-1)}}\\
\geq
\rho_{n-1}\frac{1}{N(H)} \sum_{{\ve a}\in Q(H)\cap\Z^n_{>0}} 
\left(\prod_{i=1}^{n} \frac{a_i}{\|{\ve a}\|_\infty}\right)^{1/(n-1)}- \frac{6n^2}{H^{1/(n-1)}}
\end{split}
\end{equation*}
when $H$ is large enough. Instead of summing over all $Q(H)\cap\Z^n_{>0}$ in
the first summand we will consider the subset
\begin{equation*}
{\overline Q}(H)=\left\{{\ve a}\in Q(H)\cap\Z^n_{>0} :
a_i\geq \frac{H}{2},\, 1\leq i\leq n\right \}
\end{equation*}
for which we know  $a_i/\|{\ve a}\|_\infty\geq 1/2$. In order to
estimate (very roughly) the cardinality  of  ${\overline Q}(H)$ we
start with $n=2$ and we denote this $2$-dimensional  set by ${\overline
  Q}_2(H)$. There are at most
\begin{equation*}
\left(\left\lfloor \frac{H}{m}\right\rfloor
-\left\lceil\frac{H}{2m}\right\rceil +1 \right)^2 \leq \left(\frac{H}{2m}+1\right)^2
\end{equation*}
tuples  $(a,b)\in [0,H]^2$ with $\gcd(a,b)=m$ and $a,b\geq H/2$. Thus
 \begin{equation*}
 \begin{split}
 \#{\overline Q}_2(H)  \geq \left(\frac{H}{2}\right)^2 -
 \sum_{m=2}^{H/2}\left(\frac{H}{2m}+1\right)^2
  \geq \frac{H^2}{4}\left(1-\sum_{m=2}^\infty \frac{1}{m^2}\right) -
 H \sum_{m=2}^{H/2}\frac{1}{m} -\frac{H}{2} \\
\geq  \frac{H^2}{4}\left(2-\frac{\pi^2}{6}\right) - H
\sum_{m=1}^{H/2}\frac{1}{m}
 \geq \frac{H^2}{12} - H\, (1+\ln(H/2))  \geq    \frac{H^2}{12} -2\, H\ln(H),
 \end{split}
 \end{equation*}
for $H\geq 2$. Since $\# {\overline Q}(H)\geq  \#{\overline
  Q}_2(H)\times (H/2)^{n-2}$ we get
\begin{equation*}
\begin{split}
\frac{1}{N(H)} \sum_{{\ve a}\in Q(H)\cap\Z^n_{>0}} 
\left(\prod_{i=1}^{n} \frac{a_i}{\|{\ve a}\|_\infty}\right)^{1/(n-1)}
\geq \frac{1}{(2H)^n} \sum_{{\ve a}\in \overline{Q}(H)}
\left(\prod_{i=1}^{n} \frac{a_i}{\|{\ve a}\|_\infty}\right)^{1/(n-1)} \\
\geq \frac{\# {\overline Q}(H)}{(2H)^n}\left(\frac{1}{2}\right)^{n/(n-1)} \geq
\left(\frac{1}{4}\right)^{n} \left(\frac{1}{12}-2\frac{\ln H
  }{H}\right)  \geq \left(\frac{1}{4}\right)^{n+2},
\end{split}
\end{equation*}
for $H$ large enough. Hence, all together we have found for sufficiently large  $H$
\begin{equation*}
\begin{split}
  \frac{1}{N(H)}\sum_{{\ve a}\in Q(H)} & \max_{c\in\Q^n}\frac{\gap({\ve c}, {\ve a})}{\|\ve
 {\ve a}\|_\infty^{1/(n-1)}\|\ve c\|_1}  \geq
\rho_{n-1}\left(\frac{1}{4}\right)^{n+2} - 6n^2
\frac{1}{H^{1/(n-1)}}\,. \\
\end{split}
\end{equation*}
The theorem is proved.


\begin{thebibliography}{HHHKR10}

\bibitem{lpg}
I.  Aliev, {\em On the lattice programming gap of the group problems}, Oper. Res. Lett. {\bf 43} (2015), 199–-202.

\bibitem{AG}
I. Aliev, and P. M. Gruber, {\em An optimal lower bound for the {F}robenius problem},  J. Number Theory  {\bf 123} (2007),  71--79.

\bibitem{AHO}
I. Aliev, M. Henk and T. Oertel, {\em Integrality Gaps of Integer Knapsack Problems}, Integer Programming and Combinatorial Optimization, Lecture Notes in Computer Science, {\bf 10328} (2017), 25--38.

\bibitem{Babai}
L. Babai, {\em On Lov\'asz' lattice reduction and the nearest lattice point
problem}, Combinatorica, {\bf 6} (1986), pp. 1--13.





\bibitem{Brauer}
A. Brauer, {\em On a problem of partitions},
Amer. J. Math. {\bf 64} (1942), 299–-312.

\bibitem{Cassels}
J. W. S. Cassels, {\em An introduction to the geometry of numbers}, Springer-Verlag 1971.




\bibitem{Cook}
W. Cook, A. M. H. Gerards, A. Schrijver, and \'E. Tardos, {\em Sensitivity theorems in integer linear programming}, Math. Programming {\bf 34} (1986), 251--264.

\bibitem{DF}
R. Dougherty, and V. Faber, {\em The degree-diameter problem for several varieties of Cayley graphs. I. The abelian case}, SIAM J. Discrete Math. {\bf 17} (2004), 478--519.


\bibitem{EHPS}
F. Eisenbrand, N. H\"ahnle, D. P\'alv\"olgyi, and G. Shmonin, {\em  Testing additive integrality gaps}, Math. Program. A {\bf 141}  (2013),   257--271.

\bibitem{ES}
F. Eisenbrand, and G. Shmonin, {\em Parametric integer programming in fixed dimension}, Math. Oper. Res.  {\bf 33}  (2008), 839--850.

\bibitem{EW}
F. Eisenbrand, and R. Weismantel, {\em Proximity results and faster algorithms for Integer Programming using the Steinitz Lemma},
https://arxiv.org/abs/1707.00481.


\bibitem{Fary}
I.~F\'{a}ry, {\em Sur la densit\'{e} des r\'{e}seaux de domaines
convexes}, Bull. Soc. Math. France, {\bf 78}  (1950) 152--161.







\bibitem{GLS}
M. Gr\"otschel, L. Lov\'asz and A. Schrijver, {\em Geometric algorithms and
combinatorial optimization}, Algorithms and Combinatorics vol.~2,
Springer-Verlag, Berlin, 1988.

\bibitem{peterbible}
P.M.\ Gruber, {\em Convex and discrete geometry}, Springer, Berlin, 2007.

\bibitem{GrLek}
P.M.\ Gruber, and C.G.\ Lekkerkerker, {\em Geometry of numbers}, North--Holland,
Amsterdam 1987.

\bibitem{HS}
S. Ho\c{s}ten, and B. Sturmfels,
{\em Computing the integer programming gap},
Combinatorica, {\bf 27} (2007) , no. 3, 367--382.

\bibitem{Kannan} R. Kannan, {\em  Lattice translates of a polytope and the Frobenius
problem}, { Combinatorica}, {\bf 12} (1992), 161--177.


\bibitem{MS}
J. Marklof, and A. Str\"ombergsson, {\em Diameters of random circulant graphs}, Combinatorica,  {\bf 33} (2013), 429--466.


\bibitem{Alf}
{J.~L. Ram{\'\i}rez Alfons{\'\i}n}, {\em The Diophantine Frobenius problem},
  Oxford Lecture Series in Mathematics and its Applications 30, 2005.

\bibitem{SchmidtDuke}
W. M. Schmidt, {\em Asymptotic formulae for point lattices of bounded determinant and subspaces of bounded height}, Duke Math. J., {\bf 35} (1968), 327--339.

\bibitem{Schmidt}
W. M. Schmidt, {\em Integer matrices, sublattices of $\Z^m$, and Frobenius numbers}, Monatsh. Math. {\bf 178} (2015), 405--451.

\bibitem{Schrijver}
A. Schrijver, {\em Theory of linear and integer programming}, Wiley-Interscience Series in Discrete Mathematics, 1986.

\bibitem{Str}
A. Str\"ombergsson,
{\em On the limit distribution of Frobenius numbers},
Acta Arith. {\bf 152} (2012),  81--107.


\bibitem{RT_structure}
R. R. Thomas, {\em The structure of group relaxations}, Handbooks in Operations Research and Management Science,
{\bf 12} (2005) 123--170.




\end{thebibliography}
\enddocument